\documentclass[a4paper,11pt]{amsart}
\usepackage{microtype}
\usepackage{amsmath}
\usepackage{amssymb,amsfonts,amsthm,latexsym,mathrsfs}
\usepackage{eucal}
\usepackage{color}
\input xypic
\theoremstyle{plain}
\newtheorem{theorem}[subsection]{{\bf Theorem}}
\newtheorem*{theorem*}{{\bf Theorem}}

\newtheorem*{corollary*}{{\bf Corollary}}
\newtheorem{proposition}[subsection]{{\bf Proposition}}
\newtheorem{lemma}[subsection]{{\bf Lemma}}

\newtheorem*{problem}{{\bf Problem}}

\theoremstyle{definition}

\theoremstyle{remark}

\newtheorem{example}[subsection]{{\it Example}}
\numberwithin{equation}{subsection}
\DeclareMathOperator{\PSL}{PSL}
\DeclareMathOperator{\PSU}{PSU}
\DeclareMathOperator{\Aut}{Aut}
\DeclareMathOperator{\SL}{SL}

\DeclareMathOperator{\GU}{GU}

\DeclareMathOperator{\GF}{GF}

\newcommand{\X}{\mathfrak X}
\newcommand{\C}{\mathfrak C}

\newcommand{\Z}{\mathbb Z}
\newcommand{\Q}{\mathbb Q}
\newcommand{\A}{\mathbb A}
\newcommand{\Sym}{\mathbb S}
\begin{document}
\baselineskip=12.5pt
\title[]
{Groups in which every non-cyclic subgroup contains its centralizer}
\author[C. Delizia]{Costantino Delizia}
\address{University of Salerno, Italy}
\email{cdelizia@unisa.it}
 \author[U. Jezernik]{Urban Jezernik}
\address{Institute of Mathematics, Physics, and Mechanics, Ljubljana, Slovenia}
\email{urban.jezernik@imfm.si}
\author[P. Moravec]{Primo\v z Moravec}
\address{University of Ljubljana, Slovenia}
\email{primoz.moravec@fmf.uni-lj.si}
\author[C. Nicotera]{Chiara Nicotera}
\address{University of Salerno, Italy}
\email{cnicoter@unisa.it}
\thanks{This joint work started during a visit of the first and the fourth author at the University of Ljubljana. They would like to thank the Department of Mathematics for its fine hospitality.}
\date{}
\renewcommand{\abstractname}{Abstract}
\begin{abstract}
\noindent
We study groups having the property that every non-cyclic subgroup contains its centralizer. The structure of nilpotent
and supersolvable groups in this class is described. We also classify finite $p$-groups and finite simple groups with the above defined property.
\end{abstract}
\maketitle
\section{Introduction}
\label{s:intro}
\noindent 
A subgroup $H$ of a group $G$ is {\it self-centralizing} if the centralizer $C_G(H)$ is contained in $H$. Let denote by $\C$ the class of groups in which all non-trivial subgroups are self-centralizing. It is easy to see that a finite group belongs to the class $\C$ if and only if it is either cyclic of prime order or non-abelian of order being the product $pq$ with $q<p$ primes and $p\equiv1\pmod q$. Infinite groups in $\C$ have trivial center; moreover they are periodic with each element of prime order. It easily follows that locally finite groups in $\C$ have to be finite. In \cite{Dea89, Nah93} finite groups with all elements of prime order are classified. It turns out that the primes occurring as orders of the elements of such a group are at most three. Since the class of groups with all elements of prime order is obviously quotient closed, one can conclude that  residually finite groups in the class $\C$ have finite exponent, and therefore they are finite by the Zel'manov  solution of the Restricted Burnside Problem \cite{Zel91, Zel91a}. From here, a standard argument shows that locally graded groups in the class $\C$ are finite as well. On the other hand, Tarski $p$-groups are examples of infinite groups in the class $\C$. Since groups of exponent $\leq3$ are nilpotent, for $p=2,3$ there exist no infinite $p$-groups in the class $\C$. Notice that an infinite group in the class $\C$ must contain a finitely generated infinite simple section.

Often there are also strong structural constraints on groups containing \lq\lq many\rq\rq \ self-centralizing subgroups. This in particular occurs in the case of finite $p$-groups. Shirong \cite{Shi01} proved that if $G$ is any finite nonabelian $p$-group with the property that 
$C_G(A)$ is either $A$ or $N_G(A)$ for every abelian subgroup $A$ of $G$, then the order of $G$ has to be $p^3$. On a similar note, Suzuki \cite[Proposition 1.1.8]{Ber08} showed that if a finite nonabelian $p$-group $G$ contains an abelian self-centralizing subgroup $A$ of order $p^2$, then $G$ is of maximal class. More recently, Hao and Jin \cite{Hao13}
classified all finite noncyclic $p$-groups which contain a self-centralizing cyclic normal subgroup.

In his book \cite{Ber08}, Berkovich posed the following problem:

\begin{problem}[Problem 9 of \cite{Ber08}]
Study the $p$-groups $G$ in which $C_G(A)=Z(A)$ for all non-abelian $A\le G$.
\end{problem}

The purpose of this paper is to study a relaxation of the above mentioned Berkovich's problem for not necessarily finite groups. We focus on the class $\X$ of groups $G$ with the property that if $H$ is a non-cyclic subgroup of $G$ then $C_G(H)\leq H$. Of course all cyclic groups are in the class $\X$. It is also evident that the class $\X$ is subgroup closed, and that the above defined class $\C$ is contained in the class $\X$.

Our main results are as follows. At first we show that the only abelian groups belonging to $\X$ are elementary abelian $p$-groups of order $p^2$ and Pr\"ufer $p$-group, where $p$ is a prime. Then we prove that a nilpotent group in $\X$ must be either abelian or a $p$-group of finite exponent. In the latter case we show that such a group has to be finite, and we classify all finite $p$-groups belonging to $\X$. It turns out that these are either of order $p^3$, where $p$ is odd, or $2$-groups of maximal class. 

We also deal with non-nilpotent groups in $\X$. We classify all finite nonabelian simple groups in $\X$. It turns out that these are precisely the groups $\PSL_2(q)$, where $q = 4$ or $q = 9$ or $q$ is a Fermat or a Mersenne prime.
The result does not depend on the classification of finite simple groups, since the proof only uses some classical results such as the Brauer-Suzuki theorem. 

Finally we
characterize the structure of supersolvable groups in $\X$. 
We show that these groups have to be either infinite dihedral or finite. In the finite case their structure is completely described. It mainly depends on the fact that the Sylow subgroup related to the largest prime dividing the order of the group is cyclic or non-cyclic.

\section{Nilpotent groups in the class $\X$}
\label{s:nilpotent}
\noindent 
At first we mention the following general result that will be repeatedly used throughout the paper.

\begin{lemma}
\label{l:family}
Let $\mathcal F$ be a family of finite groups, and suppose $\mathcal F$ is closed under taking subgroups. Then $\mathcal F \subseteq \X$ if and only if, for every group $G$ in $\mathcal F$, any non-cyclic subgroup of $G$ contains $Z(G)$.
\end{lemma}
\begin{proof}
Let $\mathcal F_n$ be the family of groups of order $n$ in $\mathcal F$. We prove by induction on $n$ that $\mathcal F_n \subseteq \X$. Let $G$ be a group of order $n$ in $\mathcal F$ and $H$ a non-cyclic subgroup of $G$. Let $z \in C_G(H)$. If we have $\langle H, z \rangle = G$, then $z\in Z(G)$, so $z\in H$ by hypothesis. Now suppose that $\langle H, z \rangle$ is a proper subgroup of $G$. By induction, it belongs to $\X$, so we have $z \in H$. Hence $G$ is in $\X$.

The converse is obvious.
\end{proof}

Our next goal is to characterize abelian groups in the class $\X$.

\begin{theorem}
\label{t:abelian}
Let $G$ be an abelian group. Then $G$ is in the class $\X$ if and only if one of the following holds:
\begin{itemize}
\item$G$ is cyclic;
\item$G\simeq C_p\times C_p$, the direct product of two copies of the group of prime order $p$;
\item$G\simeq\Z_{p^\infty}$, the Pr\"ufer $p$-group.
\end{itemize}
\end{theorem}
\begin{proof} Let $G$ be an abelian group in the class $\X$. Then all proper subgroups of $G$ are cyclic. If $G$ is finite and non-cyclic it easily follows that $G\simeq C_p\times C_p$. Now let us assume that $G$ is infinite. Clearly $G$ has to be periodic or torsion-free. In the former case all subgroups of $G$ are finite, so $G\simeq\Z_{p^\infty}$. In the latter case $G$ is a subgroup of the additive group $\Q$ of all rational numbers, so it has to be cyclic (see for instance \cite{Bea51}). The converse is obvious.
\end{proof}

If $G$ is any non-abelian group in the class $\X$ then of course the center $Z(G)$ of $G$ has to be cyclic.

\begin{theorem}
\label{t:nilpotent}
Let $G$ be a nilpotent group in the class $\X$. Then $G$ is either abelian or a $p$-group of finite exponent.
\end{theorem}
\begin{proof} Let $G$ be a nilpotent group in the class $\X$, and suppose that $G$ is not abelian. Thus $G$ is either periodic or torsion-free. For, if $Z(G)$ is finite then $G$ has finite exponent (see, for instance, \cite[5.2.22]{Rob82}). Otherwise $Z(G)$ is infinite cyclic, so $G$ is torsion-free.

Let $G$ be periodic. Then $G=Dr_pG_p$ is the direct product of its primary components $G_p$. Since $G$ is not abelian, the primary component $G_q$ is non-cyclic for at least one prime $q$. Thus $Dr_{p\not=q}G_p\leq C_G(G_q)\leq G_q$, therefore $G_q$ is the only non-trivial primary component of $G$.

Now assume that $G$ is torsion-free. If $H$ is a finitely generated subgroup of $G$, then $H$ has a central series $\{1\}=H_0<H_1<H_2<\dots<H_n=H$ with infinite cyclic factors (see, for instance, \cite[5.2.20]{Rob82}). Suppose $n>1$. Then $H_2$ is a finitely generated torsion-free nilpotent group in the class $\X$, and since $H_1\leq Z(H_2)$ and $H_2/H_1$ is cyclic, it follows that $H_2$ is abelian. This means that $H_2$ is cyclic, so $H_2=H_1$, a contradiction. Therefore $G$ is locally cyclic, so it is abelian.
\end{proof}

The next result provides a classification of all non-abelian $p$-groups belonging to the class $\X$.

\begin{theorem}
\label{t:finitep}
Let $G$ be a finite non-abelian $p$-group. Then $G$ belongs the class $\X$ if and only if either $p$ is odd and the order of $G$ is $p^3$, or $p=2$ and $G$ is isomorphic to one of the following groups:
\begin{itemize}
\item$D_{2^n}$, the dihedral group of order $2^n$;
\item$SD_{2^n}$, the semidihedral group of order $2^n$;
\item$Q_{2^n}$, the generalized quaternion group of order $2^n$.
\end{itemize}
\end{theorem}
\begin{proof} 
Let $G$ be a finite non-abelian $p$-group in $\X$. Suppose first that every abelian subgroup of $G$ is cyclic. It follows from \cite[Theorem 11.6]{Car99} that the group $G$ must then be isomorphic to the generalized quaternion group $Q_{2^n}$. Assume now that $G$ contains a noncyclic abelian subgroup. By Theorem \ref{t:abelian}, the latter must be of order $p^2$. Invoking \cite[Theorem 2]{Bla75}, the group $G$ is either of order $p^3$ or of maximal nilpotency class. When $p=2$, both cases reduce to $G$ being isomorphic to the dihedral, semidihedral or generalized quaternion group \cite[Corollary 3.3.4]{Lee02}. Now assume $p$ is odd and $|G| = p^n$ with $n > 3$. The group $G$ possesses a chief series of subgroups $G = P_0 > P_1 > P_2 > \cdots > P_{n-1} > P_n = 1$. The details concerning the construction may be found in \cite[Chapter 3]{Lee02}. Pick any $s \in G \backslash P_1$ and $s_1 \in P_1 \backslash P_2$, and define $s_i = [s_{i-1}, s]$ for $i \geq 2$. Consider the subgroup $P_{n-2} = \langle s_{n-2}, s_{n-1} \rangle$. As we have $[s_{n-2},s_1] = [s_{n-1},s_1] = 1$, the element $s$ centralizes $P_{n-2}$. Since the exponent of the subgroup $P_{n-2}$ equals $p$ by \cite[Proposition 3.3.2 and Corollary 3.3.4]{Lee02}, it is not cyclic. This is a contradiction with the fact that $G$ belongs to $\X$.

We now show that the admissable groups from above are indeed in the class $\X$.
Let $\mathcal F_3$ be the family consisting, for each prime $p$, of the groups of order $p^3$ and their subgroups. Let $G$ be any group in $\mathcal F_3$ and $H$ its non-cyclic subgroup. Hence $G$ is of order $p^3$ and $H$ is an abelian group of order $p^2$. Then the subgroup $H Z(G)$ of $G$ is abelian, so we have $Z(G) \leq H$. By Lemma \ref{l:family}, the family $\mathcal F_3$ is contained in $\X$.
Now let $\mathcal F_2$ be the family consisting of $2$-groups listed in the statement of the theorem, together with their subgroups. Note that each of the dihedral, semidihedral and generalized quaternion groups has exactly three maximal subgroups, all of which are either cyclic or one of the $2$-groups of maximal class. Let $G$ be any group in $\mathcal F_2$ and $H$ its non-cyclic subgroup. Recall that $G$ may be generated by two elements $x$ and $y$, and that the center $Z(G)$ is generated by the maximal non-trivial power of one of them, say $y$. Should no power of $y$ be contained in $H$, we would have $H \subseteq \{ 1, x, x y^i \}$ for some $i$, which is impossible due to $H$ not being cyclic. The subgroup $H$ thus contains a non-trivial power of $y$ and hence the center $Z(G)$. Invoking Lemma \ref{l:family}, we have $\mathcal F_2 \subseteq \X$.
\end{proof}

Now we prove that every periodic non-abelian nilpotent group in the class $\X$ is finite.

\begin{theorem}
\label{t:infinitep}
Every non-abelian nilpotent $p$-group in the class $\X$ is finite.
\end{theorem}
\begin{proof}
Let $G$ be a non-abelian nilpotent $p$-group in the class $\X$. Since $G$ is not abelian, $G$ has a finitely generated subgroup $H$ which is not cyclic. Of course $H$ is finite.

First assume that $p$ is an odd prime. Then by Theorem \ref{t:finitep} we know that the order of $H$ is $p^3$. For every element $x\in G$, the subgroup $\langle H,x\rangle$ is finite, so it has order $p^3$ too. Therefore $G$ is finite.

Now suppose that $p=2$, and assume for a contradiction that $G$ is infinite. By Theorem \ref{t:finitep} there exists a positive integer $n_1>2$ such that $H$ has order $2^{n_1}$ and nilpotency class $n_1-1$. Since $H\not=G$ there exists an element $x_1\in G$ such that $x_1\notin H$. Then the subgroup $\langle H,x_1\rangle$ has order $2^{n_2}$ and hence nilpotency class $n_2$ for some integer $n_2>{n_1}$. By iterating this argument we can construct a chain of finite subgroups of $G$ having increasing nilpotency class, a contradiction.
\end{proof}

\section{Non-nilpotent groups in the class $\X$}
\label{s:simple}
\noindent
Our first three results of this section completely describe the structure of supersoluble groups in $\X$.

\begin{theorem}
\label{t:infsupersoluble}
Let $G$ be a non-cyclic supersoluble group in the class $\X$. Then $G$ is either finite or isomorphic to the infinite dihedral group.
\end{theorem}
\begin{proof} Let $F$ denote the Fitting subgroup of $G$. Then $F$ is nilpotent, since it is finitely generated. So $F$ is finite or abelian, by Theorems \ref{t:nilpotent} and  \ref{t:infinitep}. If $F$ is finite then $G$ is also finite, since $C_G(F)$ and $G/C_G(F)$ are both finite. Now assume $F$ infinite abelian. As $F$ is finitely generated, it is cyclic by Theorem \ref{t:abelian}. Thus $G/F=G/C_G(F)$ has order 2. Write $F=\langle a\rangle$ and $G/F=\langle bF\rangle$. Clearly $G$ is not abelian by Theorem \ref{t:abelian}, so we get $G=\langle b\rangle\ltimes\langle a\rangle$ with $a^b=a^{-1}$. Moreover $b^2\in\langle a\rangle$ and thus $b^2=1$, since every non-trivial element of $\langle a\rangle$ is mapped in its inverse. Therefore $G$ is isomorphic to the infinite dihedral group.
\end{proof}

In next two statements we will assume that $G$ is a finite non-nilpotent supersoluble group and that $p>2$ is the largest prime dividing  the order of $G$. Moreover $P$ will denote a Sylow $p$-subgroup of $G$. Thus it is well-known that $P$ is normal in $G$ (see, for instance, \cite[Theorem 5.4.8]{Rob82}). Therefore $G=X\ltimes P$, where $X$ is a $p'$-group. Finally, $P_1$ will denote a minimal normal subgroup of $G$ contained in $P$. Then $|P_1| = p$, and $G/C_G(P_1)$ is cyclic of order dividing $p-1$. Of course $P_1\leq Z(P)$. We deal separately with the cases when $P$ is non-cyclic (Theorem \ref{t:ncsupersoluble}) or cyclic (Theorem \ref{t:csupersoluble}).

\begin{theorem}
\label{t:ncsupersoluble}
Assume that $P$ is non-cyclic. Then $G$ belongs to the class $\X$ if and only if $G = X \ltimes P$, where $X$ is cyclic of order dividing $p-1$ and it acts fixed point freely on $P$.
\end{theorem}
\begin{proof} Let $G$ belong to the class $\X$. Since $P$ is non-cyclic we have $C_G(P) \leq P$. As $p>2$, by Theorems \ref{t:abelian} and \ref{t:finitep} we know that $P$ is either elementary abelian of order $p^2$ or non-abelian of order $p^3$. 

Our first aim is to show that $C_G(P_1) = P$. Of course it suffices to prove that $C_G(P_1) \leq P$. If not, let $x \in C_G(P_1)$ be a $p^{\prime}$-element of $G$. First suppose $|P| = p^2$. By Maschke theorem (see, for instance, \cite[Corollary 10.17]{Isa08}) we have $P = P_1 \times S$, where $S$ is $\langle x \rangle$-invariant. Since $C_G(P) \leq P$ we get $x \notin C_G(S)$. Thus $S \langle x \rangle$ is not cyclic, so $P_1 \leq C_G(S \langle x \rangle) \leq S \langle x \rangle$ and $P\leq S \langle x \rangle$, a contradiction. Now suppose $|P| = p^3$ and $P$ non-abelian. Then $P/P_1$ is a non-cyclic normal subgroup of $G/P_1$. Let $S_1/P_1$ be a minimal normal subgroup of $G/P_1$ contained in $P/P_1$. Then $P/P_1 = S_1/P_1 \times S_2/P_1$ where $S_2/P_1$ is $\langle x \rangle$-invariant, by Maschke theorem. If both $S_1$ and $S_2$ are cyclic, then $x$ acts trivially on  every element of $S_1$ and $S_2$ having order $p$, so it acts trivially on $S_1$ and $S_2$ by \cite[Corollary 4.35]{Isa08}. Thus $x \in C_G(P) \leq P$, a contradiction. Hence we may assume that there exists $S\in\{S_1,S_2\}$ which is non-cyclic. Then, arguing as in the case $|P| = p^2$, we get $S= P_1 \times Q_1$ where $Q_1$ is $\langle x \rangle$-invariant. It follows that $x \in C_G(Q_1)$, so $x \in C_G(S) \leq S$, again a contradiction. Therefore $C_G(P_1) \leq P$ and obviously $C_G(P_1) = P$.

Hence $X=G/P$ is cyclic of order dividing $p-1$. Write $X= \langle x \rangle$. Suppose that a non-trivial power of $x$, say $x^r$, centralizes a non-trivial element $a$ of $P$. Then $P_1 \langle x^r \rangle$ is not cyclic since $C_G(P_1) \leq P$, and  
$a \in C_G(P_1\langle x^r \rangle) \leq P_1 \langle x^r \rangle $. But then $a \in P \cap P_1 \langle x^r \rangle = P_1$, a contradiction. Therefore $X$ acts fixed point freely on $P$, as required.

Conversely, using Lemma \ref{l:family} it is straightforward to show that a group $G$ with the above structure belongs to the class $\X$.
\end{proof}

\begin{example}
\label{e:ncsupersoluble}
The group
$$\left\langle a,b,c\,|\, a^2 = b^3 = c^3 = 1, [b, a] = b^2, [c, a] = c^2, [c, b] = 1\right\rangle$$
is the smallest group with the structure mentioned in Theorem \ref{t:ncsupersoluble}.  It is a split extension of the group $C_3 \times C_3$ by $C_2$. Analogously, the group
$$\left\langle a,b,c\,|\, a^3 = b^7 = c^7 = 1, [b, a] = b, [c, a] = c, [c, b] = 1\right\rangle,$$
which is a split extension of the group $C_7 \times C_7$ by $C_3$, is the smallest group of odd order having the structure of Theorem \ref{t:ncsupersoluble}.
\end{example}

\begin{theorem}
\label{t:csupersoluble}
Assume that $P$ is cyclic. Then $G$ belongs to the class $\X$ if and only if one of the following holds:
\begin{itemize}
\item[$(i)$] $G = D \ltimes C$, where $C$ is cyclic, $D$ is cyclic and $D$ acts fixed point freely on $C$;
\item[$(ii)$] $G = D \ltimes C$, where $C$ is a cyclic group of odd order, $D$ is a generalized quaternion group and $Z(G) = Z(D)$;
\item[$(iii)$] $G = D \ltimes C$, where $C$ is a cyclic $q^{\prime}$-group, $D$ is a cyclic $q$-group (here $q$ denotes the smallest prime dividing the order of $G$) and $1 < Z(G) < D$.
\end{itemize}
\end{theorem}
\begin{proof}
Let $G$ belong to the class $\X$. First, by \cite[Corollary 4.35]{Isa08} we get $C_X(P_1) = C_X(P)$, hence $C_G(P_1) = PC_X(P_1) = P \times C_X(P_1)$. Now from $P_1 \leq C_G(C_X(P_1))$ it follows that $C_X(P_1)$ is cyclic, hence $C_G(P_1)$ is cyclic. Write $C_G(P_1) = \langle c \rangle$ and $G/C_G(P_1) = \langle dC_G(P_1)\rangle$. Then $G = \langle c \rangle \langle d \rangle$, where $\langle c \rangle$ is normal in $G$.

If $Z(G) = 1$, then $\langle c \rangle \cap \langle d \rangle = 1$. In this case it is easy to see that $\langle d \rangle$ acts fixed point freely on  $\langle c \rangle$. For, if there exists $1 \not = d^n \in C_G(c^m)$ with $c^m \not = 1$ then we have $c^m \in C_G(P_1\langle d^n \rangle) \leq P_1\langle d^n \rangle$ since $d^n \notin C_G(P_1)$, thus $c^m \in \langle c \rangle \cap P_1\langle d \rangle = P_1$, a contradiction. Therefore $(i)$ holds in this case.

Now assume $Z(G) \not = 1$. Since $G$ is a non-abelian group in the class $\X$ we know that $Z(G)$ is cyclic. Let $r$ be the smallest prime dividing the order of $G/C_G(P_1)$. Then $q\leq r<p$, and any Sylow $r$-subgroup $R$ of $G$ is not contained in $C_G(P_1)$. Then $P_1R$ is non-abelian, hence $C_G(P_1R)\leq P_1R$. In particular $Z(G)\leq P_1R$, so every element $z\in Z(G)$ having prime power order $s^{\alpha}$ is contained in $P_1R$. Since the element $zP_1$ of $P_1R/P_1$ has order a power of $r$, we get $r=s$. Hence $Z(G)$ is an $r$-group. The same argument proves that $G/C_G(P_1)$ is also an $r$-group. Notice that $r=q$ is the smallest prime dividing the order of $G$. If not, any non-trivial Sylow $q$-subgroup of $G$ is contained in $C_G(P_1)$. Thus, since $C_G(P_1)$ is cyclic, there exists a normal subgroup $N$ of $G$ of order $q$. By conjugation, each element $g\in G\setminus N$ induces on $N$ an automorphism of order dividing the order of $G$. On the other hand, $\Aut(N)$ is cyclic of order $q-1$, hence $g\in C_G(N)$ and so $N\leq Z(G)$. This is a contradiction, since $Z(G)$ is a $r$-group. In particular, $Z(G)$ is properly contained in any Sylow $q$-subgroup $Q$ of $G$. Moreover, $Q$ has to be cyclic or generalized quaternion: if not, $Q$ has subgroups $Q_1\not=Q_2$ of order $q$ (see for instance \cite[5.3.6]{Rob82}), each of which must contain $Z(G)$, a contradiction. Finally, every Sylow $r$-subgroup $R$ of $G$ with $q<r$ has to be cyclic. Otherwise, with $R_1\not=R_2$ subgroups of $R$ having order $r$, the subgroup $V=R_1\times R_2$ is non-cyclic, so $Z(G)\leq C_G(V)\leq V$, a contradiction since $r\not=q$.

Now if either the order of $G$ is odd or a Sylow $2$-subgroup of $G$ is cyclic, then $(iii)$ holds where $D=Q$ and $C$ is the largest subgroup of $C_G(P_1)$ having order not divisible by $q$. On the other hand, if a Sylow 2-subgroup $Q$ of $G$ is generalized quaternion then $Z(G) = Z(Q)$ and $(ii)$ holds.

Conversely, if $(i)$ holds then $G$ is in $\X$ by Lemma \ref{l:family}. If $(ii)$ holds and $Z(G)$ has order $2^i$, then $Z(G)$ is the only subgroup of $G$ of order $2^i$. Hence if $S$ is any non-cyclic subgroup of $G$ then $2^j$ divides the order of $S$ for some $i<j$, and thus $Z(G)<S$. Similarly if $(iii)$ holds then $Z(G)$ is the only subgroup of $G$ of order 2 and again every non-cyclic subgroup of $G$ contains $Z(G)$. In both cases, the result follows by Lemma \ref{l:family}.
\end{proof}

\begin{example}
\label{e:csupersoluble}
The symmetric group $\Sym_3$ is a split extension of the group $C_3$ by $C_2$ with the action of $C_2$ being fixed point free. Clearly, it is the smallest group with the structure described in Theorem \ref{t:csupersoluble} $(i)$.

Now consider the group generated by $g_1, g_2$, $g_3$, $g_4$ subject to power relations $g_1^2 = g_2^2 = g_3^2 = g_4^3 = 1$ and commutator relations $[g_2, g_1] = g_3, [g_4, g_1] = g_4$, and $[g_i, g_j] = 1$ for the remaining $i > j$. This group is a split extension of the group $C_3$ by $Q_8$ with $Z(G_2) = Z(Q_8)$.  It is the smallest one among groups occurring in Theorem \ref{t:csupersoluble} $(ii)$.
 
Finally, the group
$$\left\langle a,b,c\,|\, a^2 = b, b^2 = c^3 = 1, [b, a] = 1, [c, a] = c, [c, b] = 1\right\rangle$$
is a split extension of the group $C_3$ by $C_4$ with $1 < Z(G_3) < C_4$. It is the smallest group with the structure described in Theorem \ref{t:csupersoluble} $(iii)$.
\end{example}

Notice that a finite group of odd order in the class $\X$ need not be supersoluble.  The smallest such non-supersoluble group is
$$\left\langle a,b,c\,|\, a^3 = b^5 = c^5 = 1, [b,a] = c^2, [c,a] = bc^2, [c,b] = 1\right\rangle,$$ which is a split extension of $C_5 \times C_5$ by $C_3$.

\begin{proposition}
\label{p:dihedral}
The dihedral group $D_{2n}$ is in the class $\X$ if and only if $n$ is odd or a power of $2$.
\end{proposition}
\begin{proof}
Denote the generators of $D_{2n}$ by $x$ and $y$ with $x^2 = y^{n} = 1$ and $y^x = y^{-1}$. If $n$ is a power of $2$, then the group $D_{2n}$ is in the class $\X$ by Theorem \ref{t:finitep}. So assume that there exists an odd prime $p$ dividing $n$. If $n$ is even, the element $y^{n/2}$ is central in $D_{2n}$, but not contained in the subgroup $\langle x, y^{n/p} \rangle$, so the group $D_{2n}$ does not belong to $\X$. Now let $n$ be odd. Note that any non-cyclic subgroup of $D_{2n}$ is again a dihedral group of order $2m$ with $m$ odd. Since the center $Z(D_{2n})$ is trivial, the result follows from Lemma \ref{l:family}.
\end{proof}

The next theorem classifies all finite simple groups in the class $\X$.

\begin{theorem}
\label{t:finitesimple}
The finite non-abelian simple groups in the class $\X$ are precisely the groups $\PSL_2(q)$, where $q = 4$ or $q = 9$ or $q$ is a Fermat or a Mersenne prime.
\end{theorem}
\begin{proof} 
Let $G$ be a finite simple group in the class $\X$, and let $S$ be a Sylow $2$-subgroup of $G$. If $S$ is abelian, then it is either cyclic or isomorphic to $C_2 \times C_2$ by Theorem \ref{t:abelian}. The cyclic case is impossible (see for instance \cite[Lemma 1.4.1]{Asc11}), whereas if $S \simeq C_2 \times C_2$, the group $G$ must be isomorphic to $\PSL_2(q)$ with $q$ odd by \cite[Chapter 15, Theorem 2.1]{Gor07}. Suppose now that $S$ is not abelian. By Theorem \ref{t:finitep}, it is isomorphic to the dihedral, semidihedral or generalized quaternion group. In the dihedral case, $G$ must be isomorphic to the group $\A_7$ or $\PSL_2(q)$ \cite[Chapter 16, Part 3]{Gor07}. If $S$ is a semidihedral group, then it follows from \cite[Third Main Theorem]{Alp70} that $G$ is isomorphic to either $M_{11}$, $\PSL_3(q)$ with $q \equiv 3 \pmod{4}$ or $\PSU_3(q)$ with $q \equiv 1 \pmod{4}$. Finally, $S$ cannot be a generalized quaternion group by the Brauer-Suzuki theorem \cite[Chapter 12, Theorem 1.1]{Gor07}.

We now verify that of the above, only the groups $\PSL(2,q)$ with $q=4,9$ or a Fermat or Mersenne prime belong to the class $\X$. First off, the alteranting group $\A_7$ contains the group $\langle (1,2)(3,4)(5,6,7), (1,3)(2,4) \rangle \cong C_6 \times C_2$, so it cannot belong to $\X$ by Theorem \ref{t:abelian}. Next, the Mathieu group $M_{11}$ contains the group $\langle (1,2,3)(4,10)(5,11,6,8,9,7), (2,3)(5,11)(6,7)(8,9) \rangle \cong D_{12}$, so it is not in $\X$ by Proposition \ref{p:dihedral}. The groups $\PSL_3(q)$ and $\PSU_3(q)$ contain as subgroups the groups $\PSL_3(p)$ and $\PSU_3(p)$. Let $g$ be the generator of the group of units of $\GF(p)$. The group $\PSL_3(p)$ contains the subgroup
\[
	\left\langle
		\left[
			\begin{array}{ccc}
			0 & 1 & 0 \\
			1 & 0 & 0 \\
			0 & 0 & -1 \\
			\end{array}
		\right],
		\left[
			\begin{array}{ccc}
			g^{p-2} & 0 & 0 \\
			0 & g & 0 \\
			0 & 0 & 1 \\
			\end{array}
		\right]
	\right\rangle,
\]
which is easily seen to be isomorphic to the dihedral group $D_{2(p-1)}$. In light of Proposition \ref{p:dihedral} and the fact that $p \equiv 3 \pmod{4}$, the group $\PSL_3(q)$ is therefore not in the class $\X$. Now consider the group $\PSU_3(p)$ with $p \equiv 1 \pmod{4}$. For $p = 5$, the group $\PSU_3(5)$ has a maximal subgroup isomorphic to $\A_7$ \cite{Con85}. The latter group contains the subgroup $\langle (1 2)(3 4)(5 6 7), (1 2)(6 7) \rangle$, which is isomorphic to the dihedral group $D_{12}$. Hence the group $\PSU_3(5)$ does not belong to $\X$. Assume now that $p > 5$. Let $N$ be the centralizer of any involution in $\PSU_3(p)$. It is shown in \cite[First Main Theorem]{Alp70} that $N$ is isomorphic to a quotient of $\GU_2(p)$ by a central subgroup of order $\gcd(p+1, 3)$. By the remark following \cite[Theorem B]{Alp70}, $N$ contains a subgroup $L_0 \cong O(N) \times \SL_2(p)$, where $O(N)$ is the largest normal subgroup of odd order of the group $N$. Denote by $r$ the largest odd number dividing $p+1$. Note that $r > 3$ since $p > 5$ and $p \equiv 1 \pmod{4}$. Let $g$ be the generator of the group of units of $\GF(p^2)$, and let $A$ be the scalar matrix in $\GU_2(p)$ with elements $g^{(p^2 - 1)/r}$ on the diagonal. The matrix $A$ is of order $r$ and the subgroup $\langle A \rangle$ is a central subgroup of odd order exceeding $3$ in $\GU_2(p)$. This means that the subgroup $O(N)$ of $N$ is not trivial. Hence $L_0$ contains an abelian subgroup $C_{r'} \times C_{2p'}$, where $r' = r/\gcd(q+1, 3)$ and $p'$ is any prime dividing $r'$. By Theorem \ref{t:abelian}, the group $\PSU_3(p)$ and hence $\PSU_3(q)$ does not belong to the class $\X$.

Finally, consider the groups $\PSL_2(q)$ with $q = p^m$. We argue as in \cite[Theorem 2]{Suz60}. Since $q$ is odd, the centralizer of any involution is a dihedral group of order $q+1$ or $q-1$ according to $q \equiv -1$ or $q \equiv 1 \pmod{4}$ \cite[Chapter 20]{Bur11}. For the group $\PSL_2(q)$ to be in $\X$, we must therefore have that either $q+1$ or $q-1$ is a power of $2$ by Proposition \ref{p:dihedral}. Suppose first that $q + 1 = 2^n$. If $m$ is even, we have $q + 1 = p^m + 1 \equiv 2 \pmod{4}$, which is impossible. So $m$ is odd any we may write $m = 2k + 1$. We then have $q + 1 = (p+1)l$ with $l$ being odd, which is only possible for $l = 1$, thus $q = p = 2^n - 1$ is a Mersenne prime. Now assume that $q - 1 = 2^n$. If $m$ is odd, the same argument as before applies to show that $q = p$ is a Fermat prime. If $m$ is even, then $q = l^2$, so $2^n = q - 1 = (l-1)(l+1)$, which is only possible for $l - 1 = 2$, and this implies $q = 9$. Note that the groups $\PSL_2(4)$ and $\PSL_2(9)$ are respectively isomorphic to the alternating groups $\A_5$ and $\A_6$. The group $\A_5$ belongs to $\X$ by Lemma \ref{l:family}, since its center is trivial and the isomorphism types of its non-cyclic subgroups are $C_2 \times C_2$, $\Sym_3$, $D_{10}$ and $\A_4$, which are all in $\X$. By the same reasoning, note that the non-cyclic subgroups of $\A_6$ are all isomorphic to $D_8$, $C_3 \times C_3$, one of the subgroups of $\A_4$, or one of the subgroups of a certain split extension of $C_3 \times C_3$ by $C_4$. These groups all belong to $\X$, so $\A_6$ does too. Now consider one of the groups $\PSL_2(p)$ with $p$ a Mersenne or a Fermat prime. Note that in case $p=3$, the group $\PSL_2(p)$ is isomorphic to $\A_4$, which belongs to $\X$. Assume now that $p > 3$. In light of Lemma \ref{l:family}, it suffices to prove that every proper subgroup of $\PSL_2(p)$ is in $\X$. The subgroup structure of $\PSL_2(p)$ is given in \cite[III, \S 8]{Hup67}, stating that any proper non-cyclic subgroup of $\PSL_2(p)$ is isomorphic to one of the following groups: the dihedral group of order $D_{2n}$ with $n$ being either a power of $2$ or odd, $\A_4$, $\Sym_4$, $\A_5$, or a subgroup of the normalizer of a Sylow $p$-subgroup that contains the Sylow subgroup itself. The dihedral groups are in $\X$ by Proposition \ref{p:dihedral}, and it is verified as above that the three permutation groups also belong to $\X$. Now consider the normalizer $N$ of a Sylow $p$-subgroup $P$ of $\PSL_2(p)$. After a possible conjugation, we may assume that $P$ is given in terms of mappings of the projective line as $P = \{ g \in \PSL_2(p) \mid s^g = s + \beta \text{ for some } \beta \in \GF(p) \}$. It is straightforward that $C_{N}(P) = P$, so any subgroup of $N$ containing $P$ must also contain the centralizer of itself in $N$. Thus $N$ is in $\X$, as required.
\end{proof}
\goodbreak

\end{document}